\documentclass[reqno, 12pt, reqno]{amsart}

\usepackage{amsfonts, amsthm, amsmath, amssymb, bbm, bm, color, enumerate,tikz-cd, eurosym, url}
\usepackage{hyperref}
\usepackage[all]{xy}
\usepackage[margin=2.5 cm]{geometry}

\usepackage{helvet}

\RequirePackage{mathrsfs} \let\mathcal\mathscr

\numberwithin{equation}{section}

\newtheorem{theorem}{Theorem}[section]

\newtheorem{lemma}[theorem]{Lemma}
\newtheorem{proposition}[theorem]{Proposition}
\newtheorem{corollary}[theorem]{Corollary}

\newtheorem{remark}[theorem]{Remark}
\newtheorem{definition}[theorem]{Definition}

\renewcommand{\phi}{\varphi}

\newcommand{\ZZ}{\mathbb{Z}}

\newcommand{\NN}{\mathbb{N}}

\newcommand{\RR}{\mathbb{R}}

\renewcommand{\emptyset}{\varnothing}

\renewcommand{\leq}{\leqslant}
\renewcommand{\geq}{\geqslant}

\newcommand{\x}{\mathbf{x}}

\renewcommand{\c}{\mathbf{c}}

\renewcommand{\b}{\mathbf{b}}
\renewcommand{\a}{\mathbf{a}}

\newcommand{\btheta}{\boldsymbol{\theta}}
\newcommand{\bgamma}{\boldsymbol{\gamma}}

\title{Integral solutions to systems of diagonal equations}
\author{Nick Rome}
\email{rome@tugraz.at}
\address{TU Graz, Institute of Analysis and Number Theory, Kopernikusgasse 24/II, 8010 Graz, Austria.}

\author{Shuntaro Yamagishi}
\email{shuntaro.yamagishi@ist.ac.at}
\address{IST Austria, Am Campus 1, 3400 Klosterneuburg, Austria.}

\begin{document}
\maketitle

\begin{abstract}
In this paper, we obtain an asymptotic formula for the number of integral solutions to a system of diagonal equations.
We obtain an asymptotic formula for the number of solutions with variables restricted to smooth numbers as well.
We improve the required number of variables compared to previous results by
incorporating recent progress on Waring's problem and the resolution of the main conjecture in Vinogradov's mean value theorem.
\end{abstract}

\section{Introduction}\label{sec:}

Consider the system of equations defined by
\begin{eqnarray}\label{eq:msos}
m_{1,1} x_{1}^d + \cdots + m_{1,n} x_n^d &=& \mu_1 \label{system} \\
                               &\vdots& \notag \\
m_{R, 1}  x_{1}^d + \cdots + m_{R, n} x_n^d &=& \mu_R, \notag
\end{eqnarray}
which we denote by $M \x^d = \boldsymbol{\mu}$,
where $M = [m_{i, j}]_{ \substack{ 1 \leq i \leq R \\ 1 \leq j \leq n } }$ is the coefficient matrix with integer entries,
$\x^d = \begin{bmatrix}
          x_1^d \\
          \vdots  \\
          x_n^d
        \end{bmatrix}$
and
$\boldsymbol{\mu} = \begin{bmatrix}
          \mu_1 \\
          \vdots  \\
          \mu_R
        \end{bmatrix} \in \ZZ^R.$
The system of diagonal equations (\ref{system}) with $\boldsymbol{\mu} = \mathbf{0}$ was first studied by Davenport and Lewis \cite[Lemma 32]{DavenportLewis}
who established the following.
\begin{theorem}[Davenport and Lewis]
Let $d \geq 3$ and $\boldsymbol{\mu} = \mathbf{0}$. Suppose that all $n$ variables occur explicitly in the equations (\ref{system}). Suppose that any linear combination, not identically zero, of
the $R$ rows of $M$ contains more than $(2 H + 3d - 1)R$ non-zero entries, where $H = \lfloor 3 d \log R d\rfloor$. Suppose the equations (\ref{system}) have a non-singular solution in every $p$-adic field, and further, if $d$ is even, a real non-singular solution. Then the equations (\ref{system}) have infinitely many solutions in integers.
\end{theorem}
In fact, they obtained an asymptotic formula for the number of solutions.
Their main results \cite[Theorems 1 and 2]{DavenportLewis} are consequences of this theorem and require
$$
n \geq \begin{cases} \lfloor 9 R^2 d \log (3 R d) \rfloor & \mbox{if $d$ is odd}, \\
\lfloor 48 R^2 d^3 \log (3 R d^2) \rfloor & \mbox{if $d \geq 4$ is even},
\end{cases}
$$
for the conclusions to hold. By incorporating the breakthrough on Waring's problem by Vaughan \cite{Va89}, Br\"{u}dern and Cook \cite{BC}
improved the number of variables required to
$$
n > n_0(d) R,
$$
where $n_0(d) = 2 d (\log  d + O(\log \log d) )$, under a suitable ``rank condition'' on the coefficient matrix $M$.
They also obtained an asymptotic formula for the number of solutions but with variables restricted to smooth numbers, which  in turn provided a lower bound for the number of solutions in positive integers.

Since the release of these two papers, there has been great progress regarding Waring's problem
(for example, by Wooley \cite{Woooo92}, \cite{Woo95} and more recently by Wooley and Br\"{u}dern \cite{BruWoo})
and also the resolution of the main conjecture in Vinogradov's mean value theorem (see the work by Bourgain, Demeter and Guth \cite{BDG},
and by Wooley \cite{WooCub, Wooley}). The purpose of this paper is to incorporate these recent progress to improve the required number of variables 
in both the setting of solutions in positive integers as in \cite{DavenportLewis}
and in the smooth numbers as in \cite{BC}. 

There have been a number of results regarding pairs of diagonal equations, 
in which the improvements have been achieved by making use of various developments in the theory of smooth Weyl sums.    
For example,  the work of Parsell \cite{parsellpairs} on pairs of equations of small degrees,
Parsell and Wooley \cite{quintic} on pairs of quintic equations, 
and 
Br{\"u}dern and Wooley \cite{cubicpairs} on pairs of cubic equations. 
For larger systems of diagonal equations, there are the papers of  
Br{\"u}dern and Wooley \cite{cubicsystem} on systems of cubic equations,
and 
of  Brandes and Parsell \cite{brandesparsell} 
and Brandes and Wooley \cite{JuTr}
on systems of equations involving different degrees. These works assume that the system is 
``highly non-singular'' which is to say that any $R \times R$ submatrix of the coefficient matrix is invertible.  
Our work is instead in line with \cite{DavenportLewis} and \cite{BC} which hold for systems of diagonal equations of the same (arbitrary) degree with slightly less restrictive conditions on the underlying coefficient matrix.

For $X \geq 1$ and $\mathfrak{B} \subseteq \NN$, we introduce the following counting function
\begin{align*}
N( \mathfrak{B}; X) = \#\{\x \in ( \mathfrak{B} \cap [1,X]  )^n:  M \x^d = \boldsymbol{\mu}  \}.
\end{align*}
Instead of restricting the singularity of the variety defined by our system of equations,
as in the work of Birch \cite{birch},
we will require a condition on how well the underlying coefficient matrix can be partitioned.
\begin{definition}
\label{TTT}
For an $R \times n$ matrix $A$ with $n \geq R$, we define $\Psi (A)$ to be the largest integer $\mathfrak{T}$ such that there exists
$$
\{ \mathfrak{D}_1, \ldots, \mathfrak{D}_\mathfrak{T} \},
$$
where each $\mathfrak{D}_i$ is a linearly independent set of $R$ columns of $A$ and $\mathfrak{D}_i \cap \mathfrak{D}_j \neq \emptyset$ if $i \neq j$.
\end{definition}

\begin{remark}
There are at least two ways to obtain lower bounds for $\Psi(M)$: by studying the ranks of submatrices (thanks to a result of Low--Pitman--Wolff \cite[Lemma 1]{LPW}) as in \cite{flores}, or by algorithmically enumerating sets of $R$ linearly independent columns as in \cite{msos}. In these two papers,
lower bounds of the form a constant times $\frac{n}{R}$ were obtained for coefficient matrices related to $n \times n$ magic squares.
\end{remark}

The following are the main results of this paper.
\begin{theorem}
\label{main2}
Let $d \geq 2$ and  $T_{\text{int.}}(d)$ be as recorded in Table 1.
If $\Psi(M) \geq T_{\text{int.}}(d) + 1$,
then there exists $\gamma > 0$ for which
\begin{eqnarray}
\notag
N(\NN; X) =  \mathfrak{S} \mathfrak{I} X^{n - d R} + O(  X^{n  - d R - \gamma} ),
\end{eqnarray}
where $\mathfrak{S}$ is the singular series defined in (\ref{singser}) and $\mathfrak{I}$ is the singular integral defined in (\ref{singint}).
We remark that $T_{\text{int.}}(d) \leq \min \{2^d, d(d+1) \}$ for all $d \geq 2$.
\end{theorem}

Given $1 \leq Z \leq X$, we denote the $Z$-smooth numbers by
$$
\mathcal{A}(X, Z) = \{ x \in [1, X] \cap \ZZ: \textnormal{prime } p| x \textnormal{ implies } p \leq Z  \}.
$$

\begin{theorem}
\label{main2+}
Let $d \geq 5$ and $T_{\text{smo.}}(d)$ be as recorded in Table 2.
If $\Psi(M) \geq T_{\text{smo.}}(d) + 1$, then for $\eta > 0$ sufficiently small, there exists $\gamma > 0$ such that
\begin{eqnarray}
\notag
N(\mathcal{A}(X, X^{\eta}) ; X) = c(\eta)^n \mathfrak{S} \mathfrak{I} X^{n - d R} + O(  X^{n  - d R} (\log X)^{ - \gamma} ),
\end{eqnarray}
where $\mathfrak{S}$ is the singular series defined in (\ref{singser}), $\mathfrak{I}$ is the singular integral defined in (\ref{singint})
and $c(\eta) > 0$ depends only $\eta$. We remark that $T_{\text{smo.}}(d) \leq  \lceil d(\log d + 4.20032)\rceil$ for all $d \geq 5$.
\end{theorem}

\begin{remark}
Instead of the condition $\Psi(M) \geq T_{\text{int.}}(d) + 1$ in Theorem \ref{main2},
we may assume that there exists an $R \times  (R T_{\text{int.}}(d) + 1)$ submatrix of $M$ with the property that after removing any one of its columns it still contains
$T_{\text{int.}}(d)$ pairwise disjoint $R \times R$ invertible submatrices; the same holds for  Theorem \ref{main2+} with $T_{\text{smo.}}(d)$ in place of $T_{\text{int.}}(d)$.
This is essentially the hypothesis assumed in \cite{BC}, but for simplicity we assume the former condition; to assume the latter condition, one needs to slightly modify  the proof of
Proposition \ref{prop:minor}.
\end{remark}

An immediate corollary is a lower bound for  $N(\NN; X)$ which requires a smaller value of $\Psi(M)$  than in Theorem \ref{main2+}.
\begin{corollary}\label{cor:lowerbound}
Let $d \geq 3$ and suppose $\Psi(M) \geq T_{\text{smo.}}(d) + 1$.
Then for a fixed $\boldsymbol{\mu}$ such that $\mathfrak{SI} >0$, we have
\begin{eqnarray}
\notag
N(\NN; X) \gg   X^{n - d R}.
\end{eqnarray}
\end{corollary}

We note that for a fixed choice of $\boldsymbol{\mu}$, by standard arguments, $\mathfrak{S} > 0$ if the equations (\ref{system}) have a non-singular solution in every $p$-adic field, and $\mathfrak{I} > 0$ if the equations (\ref{system}) have a real non-singular solution.
\newline

\begin{tabular}{ c c }
Table 1 & Table 2\\
\begin{tabular}{|c|c|}
\hline
$d$ & $T_{\text{int.}}(d)$ \\
\hline
2 & 4 \\
3 & 8 \\
4 & 15 \\
5& 23 \\
6 & 34 \\
7 & 47\\
8 & 61 \\
9 & 78\\
$\geq 10$ & $d^2 - d + 2 \lfloor \sqrt{2d+2} \rfloor - \theta(d) $\\
\hline
\end{tabular} &
%
%
%
 \begin{tabular}{|c|c|| c | c|}
\hline
$d$ & $T_{\text{smo.}}(d)$ & $d$ &$T_{\text{smo.}}(d)$ \\
\hline
5 &  19  & 13 & 81 \\
6&  25    & 14 & 89\\
7 & 33 & 15 & 97 \\
8 & 41  & 16 & 105\\
9 & 49 & 17 & 113\\
10 & 57 & 18 & 121\\
11 & 65 & 19 & 129 \\
12 & 73 & $\geq 20$ & $\lceil d(\log d + 4.20032)\rceil$ \\
\hline
\end{tabular} \\
%
\end{tabular}
\newline
\newline

The function $\theta(d) \in \{1, 2\}$ which appears in Table 1 is defined in (\ref{deftheta}).
The values of $T_{\text{int.}}(d)$ are described in Lemma \ref{lem:mvt}, and
they correspond to the smallest known number of variables $s$ required to produce an asymptotic formula for the number of representations of any sufficiently large natural number as a sum of $s$ $d^\text{th}$ powers (compare with \cite[Cor 14.7]{Wooley} for larger powers and \cite[Thm 4.1]{imrn} for intermediate powers).

For $d \geq 13$, the values of $T_{\text{smo.}}(d)$ correspond to the best known values of $G(d)$, the least number of variables required to represent every sufficiently large natural number as a sum of $d^\text{th}$ powers. Note the distinction that in this problem one asks only for the existence of a solution, and
not the asymptotic formula for the number of solutions. For smaller values of $d$, $T_{\text{smo.}}(d)$ is slightly larger than the best known values of $G(d)$,
which are $G(7) \leq 31, G(8) \leq 39, G(9) \leq  47, G(10)   \leq 55, G(11)\leq 63$ and $G(12)\leq 72$ as found in \cite{intermediate});
these values are obtained by considering solutions to the underlying Diophantine equations for which only some of the variables are restricted to the smooth numbers.

\subsection*{Acknowledgements}
NR was supported by FWF project ESP 441-NBL while SY by a FWF grant (DOI 10.55776/P32428).
The authors are grateful to J\"{o}rg Br\"{u}dern for helpful discussions regarding his paper \cite{BC}
and to Trevor Wooley for numerous helpful comments on an earlier version of this paper and on the state of the art in Waring's problem.

\subsection*{Notation}
We make use of the standard abbreviations $e(z) = e^{2\pi i z}$ and $e_q(z) = e^{\frac{2 \pi i z}{q}}$.
Given a vector $\a = (a_1, \ldots, a_R) \in \ZZ^R$, by $0 \leq \a \leq q$ we mean $0 \leq a_i \leq q$ for each $1 \leq i \leq R$.
We also let $|\bgamma| = \max_{1 \leq i \leq R} |\gamma_i|$ for any  $\bgamma \in \RR^R.$

\section{Preliminaries}
%
\subsection{Weyl sums}
In this section, we collect two results which are the main ingredients to prove Theorem \ref{main2}.
Both are consequences of the resolution of the main conjecture in Vinogradov's mean value theorem (by Bourgain, Demeter and  Guth \cite{BDG} and by Wooley \cite{WooCub, Wooley}).
\begin{lemma}\label{lem:weyl}
Let $d \geq 2$. Let $\alpha \in \RR$ and suppose that there exist $q \in \NN$ and $a \in \ZZ$  with $\gcd(q, a)=1$ such that $\vert \alpha - a/q \vert \leq q^{-2}$ and $q \leq X^d$.
We define
\[
\lambda(d) =
\begin{cases}
\frac{1}{2^{d-1}} & \mbox{if } 2 \leq d \leq 5, \\
\frac{1}{d(d - 1)}& \mbox{otherwise.}
\end{cases}
\]
Then
\[
\left \vert \sum_{1 \leq  x \leq X} e(\alpha x^d) \right \vert
\ll
X^{1+\varepsilon} ( q^{-1} + X^{-1} + q X^{-d})^{\lambda(d)},
\]
for any $\varepsilon > 0$.
\end{lemma}
\begin{proof}
The bound for $2 \leq d \leq 5$ is the classic Weyl's inequality \cite[Lemma 2.4]{Vaughan}.
The other estimate for larger $d$ is a consequence of the resolution of the main conjecture in Vinogradov's mean value theorem (c.f. \cite[Lemma 2.4]{Campana}).
\end{proof}

Let us define
\begin{eqnarray}
\label{deftheta}
\theta(d) = \begin{cases}
              1 & \mbox{if } 2d + 2 \geq \lfloor  \sqrt{2d + 2} \rfloor^2 + \lfloor  \sqrt{2d + 2} \rfloor,   \\
              2 & \mbox{if } 2d + 2 < \lfloor  \sqrt{2d + 2} \rfloor^2 + \lfloor  \sqrt{2d + 2} \rfloor.
            \end{cases}
\end{eqnarray}
\begin{lemma}\label{lem:mvt}
Let $d \geq 2$ and $s$ a real number such that $$s \geq T_{\text{int.}}(d) = \min\left\{2^d,d^2 - d + 2 \lfloor \sqrt{2d+2} \rfloor - \theta(d),  d^2 +1 - \max \limits_{\substack{1 \leq j \leq d-1 \\ 2^j \leq d^2}} \left \lceil \frac{dj - 2^j}{d+1 - j} \right \rceil \right\}.$$ Then
\[
\int_0^1 \left \vert \sum_{1 \leq x \leq X} e(\alpha x^d) \right \vert^s \mathrm{d} \alpha \ll X^{s - d + \varepsilon},
\]
for any $\varepsilon > 0$.
\end{lemma}

\begin{proof}
The first bound $s \geq 2^d$ is the classical version of Hua's lemma \cite[Lemma 2.5]{Vaughan}, while
the other two bounds are consequences of the resolution of the main conjecture in Vinogradov's mean value theorem.
The second bound can be found in \cite[Cor. 14.7]{Wooley}, on noting that the bound for $s_0$ in the statement is given by
$$
s_0 \leq \lfloor s_0 \rfloor +  1 \leq d^2 - d + 2 \lfloor \sqrt{2d+2} \rfloor - \theta(d)
$$
as explained in the proof. The third bound essentially follows from \cite[Thm. 4.1]{imrn}; it can be seen in the proof that the
integral over the minor arcs satisfies $\ll X^{s - d + \varepsilon}$, while over the major arcs the same estimate follows
by combining familiar estimates from the major arc analysis in the theory of Waring's problem (see \cite[Section 4]{Vaughan}).
It can be verified that the values of $T_{\text{int.}}(d)$ are precisely as in Table 1 (see the paragraph following \cite[Cor. 14.7]{Wooley}
and the proof of \cite[Thm. 4.1]{imrn}).
\end{proof}

\subsection{Smooth Weyl sums}
\label{smoothW}
In this section, we record some key estimates regarding the smooth Weyl sums needed to prove Theorem \ref{main2+}.
Let $d \geq 3$. We let
$$
f(\alpha; X, Z) = \sum_{ x \in \mathcal{A}(X, Z) } e(\alpha x^d).
$$

We first need two estimates from \cite{BC}.
We begin with \cite[Lemma 3]{BC} which is obtained by combining \cite[Theorem 1.8]{Va89} and \cite[Lemma 7.2]{VW}.
\begin{lemma}\cite[Lemma 3]{BC}
\label{BClem}
Let $d \geq 3$ and  $\varepsilon > 0$ be sufficiently small.
Suppose $\eta > 0$ is sufficiently small.
Then there exists $\gamma = \gamma(d) > 0$ such that
given $\alpha \in [0,1]$ one of the following two alternatives holds:
\begin{enumerate}
\item[(i)] we have $| f(\alpha; X, X^{\eta}) | < X^{1 - \gamma}$;
\item[(ii)] there exist $0 \leq a \leq q$, $\gcd(q, a) = 1$ such that
$$
f(\alpha; X, X^{\eta})  \ll q^{\varepsilon} X  ( q + X^d | q \alpha - a| )^{ -  \frac{1}{2d}} (\log X)^3.
$$
\end{enumerate}
\end{lemma}

The following is \cite[Lemma 4]{BC}, which is a special case of \cite[Lemma 8.5]{VW}.
\begin{lemma}\cite[Lemma 4]{BC}
\label{BClem'}
Let $d \geq 3$.
Suppose $\eta > 0$ is sufficiently small.
Let $A_0 > 0$. Suppose $\gcd(q,a) = 1$, $1 \leq q \leq (\log X)^{A_0}$ and $|q \alpha - a| \leq (\log X)^{A_0} X^{-d}$.
Then
$$
f(\alpha; X, X^{\eta}) \ll X q^{\varepsilon} (q + X^d |q\alpha - a|)^{- \frac{ 1 }{ d } },
$$
for any $\varepsilon > 0$.
\end{lemma}

Given a real parameter $\mathfrak{L} \geq 1$, we define
$$
\mathfrak{N}_\mathfrak{L} = \bigcup_{1 \leq q \leq  \mathfrak{L}} \bigcup_{ \substack{ 0 \leq a \leq q  \\  \gcd(q, a) = 1  }  }
\{ \theta \in [0,1]: \vert q \theta - a \vert <  \mathfrak{L} X^{ - d} \}.
$$
We make use of the previous two lemmas to prove the following.
\begin{lemma}
\label{min+}
Let $\delta > 0$,  $A = 2 d \delta$ and $\mathfrak{L} = (\log X)^A$.
Suppose $\eta > 0$ is sufficiently small.
If
$$
|f(\alpha; X, X^{\eta}) | >  X (\log X)^{- \delta}
$$
holds for $X \geq 1$ sufficiently large,
then
$$
\alpha \in \mathfrak{N}_\mathfrak{L}.
$$
\end{lemma}
\begin{proof}
Since we are in alternative (ii) of Lemma \ref{BClem}, it follows that
$$
X (\log X)^{- \delta} < C q^{\varepsilon} X ( q + X^d | q \alpha - a| )^{ - \frac{1}{2d}} (\log X)^3,
$$
for  $\varepsilon > 0$ sufficiently small and  some $C > 0$, which in turn implies
$$
q^{\frac{1}{2d}}   < C  q^{\varepsilon}  (\log X)^{\delta + 3}
$$
and
$$
(  X^d | q \alpha - a| )^{  \frac{1}{2d}}  < C q^{\varepsilon} (\log X)^{\delta + 3}.
$$
Therefore, by setting $A_0 =  (\delta + 3) 4 d$, we obtain
$1 \leq  q <  (\log X)^{ 4 d(\delta + 3) }$, $\gcd(q, a) = 1$ and $|q \alpha - a| < (\log X)^{A_0} X^{-d}$.
It then follows from Lemma \ref{BClem'} that
$$
X (\log X)^{- \delta} < C_1 q^{\varepsilon} X (q + X^d |q \alpha - a|)^{- \frac{1}{d}},
$$
for some $C_1 = C_1(d, \delta, \varepsilon) > 0$, which in turn implies
$$
q^{\frac{1}{d}}   < C_1  q^{\varepsilon}  (\log X)^{\delta}
$$
and
$$
(  X^d | q \alpha - a| )^{  \frac{1}{d}}  < C_1 q^{\varepsilon} (\log X)^{\delta}.
$$
Therefore, for $\mathfrak{L} = (\log X)^A$ with $A = 2 d \delta$, it follows that
$\alpha \in \mathfrak{N}_\mathfrak{L} $
as desired.
\end{proof}

Finally, we have the following mean value estimate from \cite{BruWoo}.
\begin{lemma}
\label{MVT1}
Let $d \geq 5$ and $s$ be an integer such that $s \geq T_{\text{smo.}}(d)$ as recorded in Table 2.
Let $\eta > 0$ be sufficiently small and $1 \leq Z \leq X^{\eta}$.
Then
$$
\int_0^1 |f(\alpha; X, Z)|^s \mathrm{d} \alpha \ll X^{s - d}.
$$
\end{lemma}

\begin{proof}
%
A real number $\Delta_s$ is referred to as an admissible exponent (for $d$) if it has the property that, whenever $\varepsilon > 0$ and $\eta$ is a positive number sufficiently small in terms of $\varepsilon$, $d$ and $s$, then whenever $1 \leq Z \leq X^\eta$ and $X$ is sufficiently large, one has
$$
\int_0^1 |f(\alpha; X, Z)|^s \mathrm{d} \alpha \ll X^{s - d + \Delta_s + \varepsilon}.
$$
Let us introduce the number
$$
\tau(d) = \max \limits_{w \in \NN} \frac{d - 2 \Delta_{2w}}{4 w^2}.
$$
Suppose that $s$ is a real number with $s \geq 2$, and that the exponents $\Delta_u$ are admissible for $2 \leq u \leq s$. We define
\[
\Delta_s^* = \min \limits_{0 \leq t \leq s-2}  ( \Delta_{s-t} - t \tau (d) ),
\]
and refer to $\Delta_s^* $ as an admissible exponent for minor arcs.
Let $d \geq 3$, $s \geq 2d + 3$  and $\Delta_s^*$ is an admissible exponent for minor arcs
with $\Delta_s^*  < 0$. Then applying \cite[Theorem 6.1]{BruWoo} with $Q =1$ provides the bound
\[
\int_0^1 |f(\alpha; X, Z)|^s \mathrm{d} \alpha \ll X^{s - d}.
\]

We now follow the argument in the proof of \cite[Theorem 6.2]{BruWoo}.
We assume that we have available an admissible exponent $\Delta_u$ for each positive number $u$  (which we know we may assume as explained in \cite[Section 2]{BruWoo}, and also see \cite[(7.1)]{BruWoo} for further information regarding $\Delta_u$ when $u$ is even and $d \geq 4$).
When $d \geq 4$, we define
\begin{eqnarray}
\label{6.11}
G_0(d)  =  \min_{v \geq 2} \left( v + \frac{\Delta_v}{ \tau(d) } \right).
\end{eqnarray}
Suppose that $d \geq 4$ and $s \geq \max \{ \lfloor G_0 (d) \rfloor + 1, 2d + 3 \}$.
Then there exists a positive number $v$ with $v \geq 2$ and an admissible exponent $\Delta_v$
for which the exponent $\Delta_s^*$ is admissible for minor arcs, where
$$
\Delta^*_s = \Delta_v - (s - v) \tau(d) = - \tau(d) (s - G_0(d)) < 0.
$$

For $d \geq 14$, the value of $T_{\text{smo.}}(d)$ is precisely the value of $\lfloor G_0(d) \rfloor + 1$ found in the proofs of \cite[Thm 1.1 and Thm 1.3]{BruWoo}, which
can be seen to be greater than $2d + 3$. For smaller $d$, we follow the proof of \cite[Theorem 8.1]{BruWoo} and compute $G_0(d)$ using the expression
$$
T(d) =  \frac{4 w^2}{d - 2 \Delta_{2w}}
$$
for a suitably chosen value of $w$. Since $\tau(d) \geq T(d)^{-1}$, we clearly have
$$
G_0(d) \leq v' + \Delta_{v'} T(d)
$$
for any choice of $v' \geq 2$.
We use the values of $w$ and the corresponding admissible exponents
$\Delta_{2w}$ recorded in Vaughan--Wooley \cite[$\S$9--15]{FurtherImprovements4}. Here, the exponents
$\lambda_w$ of \cite{FurtherImprovements4} are related to $\Delta_{2w}$ via the formula
$\Delta_{2w} = \lambda_w - 2w + d$.
Below we record the chosen values of $w$ and $v$ used to compute $2v + \Delta_{2v} T(d)$.

\begin{center}
\begin{tabular}{|c|c|c|c|c|c|c|}
\hline
$d$ & $w$ & $\Delta_{2w}$ & $T(d)$ & $v$ & $\Delta_{2v}$ & $ 2v + \Delta_{2v} T(d) $ \\
\hline
7 & 6 &  2.0143820 & 48.46467935 & 16 & 0.0105382 & 32.51073048 \\
8 & 7 &  2.3105992 & 58.00873304 & 19 & 0.0473193 & 40.74493264 \\
9 & 8 &  2.6039271 & 67.50795289 & 22 & 0.0727119 & 48.90863152 \\
10 & 9 & 2.8945712 & 76.94394605 & 25 & 0.0895832 & 56.89288491 \\
11& 10 & 3.1849727 & 86.39206976 & 28 & 0.1020502 & 64.81632800 \\
12 & 11& 3.4700805 & 95.65521749 & 31 & 0.1118679 & 72.70074830 \\
13 & 12 &3.7557170 &104.94544480 & 35 & 0.1010835 & 80.60825287 \\
\hline
\end{tabular}
\end{center}
For $d=5$ and $6$, the necessary data come instead from the appendix of \cite{FurtherImprovements} and we choose the following values.
\begin{center}
\begin{tabular}{|c|c|c|c|c|c|c|}
\hline
$d$ & $w$ & $\Delta_{2w}$ & $T(d)$ & $v$ & $\Delta_{2v}$ & $ 2v + \Delta_{2v} T(d) $ \\
\hline
5 & 4 & 1.4386563 & 30.15045927  &  8 & 0.0773627 & 18.33252094  \\
6 & 5 & 1.7246965 & 39.20635362  & 12 & 0.0000000 & 24.00000000  \\
\hline
\end{tabular}
\end{center}
One readily observes that
$$
T_{\text{smo.}}(d) = \lfloor 2v + \Delta_{2v} T(d) \rfloor + 1 \geq  \lfloor G_0(d) \rfloor + 1
$$
for the values of $d$ listed in the tables above.
\end{proof}

\section{The Hardy-Littlewood circle method}
Let $\mathfrak{B} = \mathbb{N}$ or $\mathcal{A}(X, X^{\eta})$. Throughout the remainder of the paper, unless stated otherwise, we assume $d \geq 2$ if $\mathfrak{B} = \mathbb{N}$, and $d \geq 3$ if $\mathfrak{B} = \mathcal{A}(X, X^{\eta})$.
Our main tool
to study $N (\mathfrak{B}; X)$
is the Hardy--Littlewood circle method and
the key input are the estimates regarding the associated exponential sums.
In contrast to the exposition in \cite{BC}, we find it more natural to index our exponential sums by the columns of the corresponding coefficient matrix.
For $\btheta \in [0,1]^{R}$ and $\c \in \textnormal{Col}(M)$, we introduce the exponential sum
\[
S_{\c}(\btheta) = S_{\c}( \mathfrak{B}; \btheta) =  \sum_{ x \in \mathfrak{B} \cap [1, X] } e(\c \cdot\btheta x^d).
\]
Then
\begin{eqnarray}
\label{circle...}
N (\mathfrak{B}; X) = \int_{[0,1]^{R}} \prod_{\c \in \textnormal{Col} (M) } S_{\c}( \mathfrak{B}; \btheta) \cdot e\left( - \sum_{i = 1}^{R} \mu_i \theta_i   \right) \mathrm{d}\btheta.
\end{eqnarray}
We set
\begin{eqnarray}
\label{defL}
\mathfrak{L} = \begin{cases}
                 X^{\delta} & \mbox{if } \mathfrak{B} = \mathbb{N},  \\
                 (\log X)^A & \mbox{if } \mathfrak{B} = \mathcal{A}(X, X^{\eta}),
               \end{cases}
\end{eqnarray}
where $\delta, A > 0$ are to be chosen in due course.
We define the major arcs
$$
\mathfrak{M}_\mathfrak{L} = \bigcup_{1 \leq q \leq  \mathfrak{L} } \bigcup_{ \substack{ \a \in \ZZ^{R} \\ 0 \leq \a \leq q  \\  \gcd(q, \a) = 1  }  }
\{ \boldsymbol{\theta} \in [0,1]^{R}: \vert q \theta_i - a_i \vert <  \mathfrak{L} X^{ - d} \ (1 \leq i \leq R) \},
$$
and the minor arcs
$$
\mathfrak{m}_\mathfrak{L} = [0,1]^{R} \setminus \mathfrak{M}_\mathfrak{L}.
$$
From here on out, we will use the following notation for simplicity.
\begin{definition}
We let $T$ be a natural number such that $\Psi(M) \geq T$.
\end{definition}


\subsection{The minor arc estimate}
The following lemma allows us to understand when a phase of the form $\c\cdot\btheta$  belongs to $[0,1] \setminus \mathfrak{N}_\mathfrak{L}$.
Given a set of vectors $\mathfrak{D} = \{ \c_1, \ldots, \c_{R} \}$, we denote by $M( \mathfrak{D} ) = [ \c_1 \cdots \c_{R} ]$
the matrix with these vectors as its columns.

\begin{lemma}\label{lem:linalg}
Let $\mathfrak{D} = \{ \c_1, \ldots, \c_{R} \}  \subseteq \textnormal{Col}( M )$ be a set of $R$ linearly independent vectors.
Suppose $X \geq 1$ is sufficiently large.
If $\c_i \cdot \btheta \in \mathfrak{N}_{\mathfrak{L}^{1/(R+1)} }$ for all $1 \leq i \leq R$,
then $\btheta \in \mathfrak{M}_\mathfrak{L}$.
\end{lemma}

\begin{proof}
We have
$$
\begin{bmatrix}
                                           q_1 \c_1 \cdot \btheta  \\
                                           \vdots  \\
                                           q_R \c_R \cdot \btheta
                                         \end{bmatrix}
   =  \begin{bmatrix}
                                           a_1 + E_1 \\
                                           \vdots  \\
                                           a_{R} + E_{R}
                                         \end{bmatrix}
$$
for some $1 \leq q_i \leq \mathfrak{L}^{1/(R+1)}$ and $1 \leq a_i \leq q_i$ such that $\gcd(a_i, q_i) = 1$
and $|E_i| <  \mathfrak{L}^{1/(R+1)} X^{- d}$ for each $1 \leq i \leq R$. Then
$$
q_1 \cdots q_R M( \mathfrak{D} )^{t} \btheta
   =  \begin{bmatrix}
                                            q_1 \cdots q_R (a_1 + E_1) / q_1 \\
                                           \vdots  \\
                                           q_1 \cdots q_R (a_{R} + E_{R}) /q_R
                                         \end{bmatrix},
$$
and the result follows
by multiplying both sides of the equation by the inverse of $M( \mathfrak{D} )^{t}$ on the left
and simplifying the resulting equation.
\end{proof}

We are now ready to bound the contribution from the minor arcs.
\begin{proposition}\label{prop:minor}
Suppose that
$$
T \geq
\begin{cases}
                  T_{int.}(d) + 1 & \mbox{if }  \mathfrak{B} = \mathbb{N},  \\
                  T_{smo.}(d) + 1  & \mbox{if }  \mathfrak{B} = \mathcal{A}(X, X^{\eta}).
               \end{cases}
$$
Suppose $\eta > 0$ is sufficiently small.
Then, we may choose $\delta, A > 0$ such that there exists $\gamma > 0$ satisfying
$$
\int_{\mathfrak{m}_\mathfrak{L}}  \prod_{\c \in \textnormal{Col}(M)} |S_{\c}(\mathfrak{B};  \boldsymbol{\theta})|  \mathrm{d}\boldsymbol{\theta}
\ll
X^{n - d R} \mathfrak{L}^{- \gamma}.
$$
\end{proposition}

\begin{proof}
Let
$$
\mathfrak{D}_1, \ldots, \mathfrak{D}_{ T }
$$
be pairwise disjoint sets of $R$ linearly independent columns of $M$.
We begin by applying Lemma \ref{lem:linalg} with $\mathfrak{D}_{T } = \{ \c_1, \ldots, \c_{R} \}$.
Given $\btheta \in \mathfrak{m}_\mathfrak{L}$, it follows from Lemma \ref{lem:linalg} that there exists $1 \leq i \leq R$ such that $\mathbf c_i \cdot \btheta \not \in \mathfrak{N}_\mathfrak{L'}$ with $\mathfrak{L'} = \mathfrak{L}^{1/(R+1)}$. Extracting the contribution from this column, we have the bound
\begin{eqnarray}
\notag
&& \int_{ \mathfrak{m}_\mathfrak{L} } \prod_{\c \in \textnormal{Col}(M)  } |S_{\c}(\btheta)| \mathrm{d}\btheta
\\
\notag
&\leq& X^{R-1}  \sup_{\alpha \in [0,1] \setminus  \mathfrak{N}_{\mathfrak{L}'}}  \left|  \sum_{x \in \mathfrak{B} \cap [1,X]} e (\alpha x^d) \right|
\int_{[0,1]^R} \prod_{\substack{ \mathbf c \in \text{Col}(M)  \setminus  \mathfrak{D}_{ T  } }} |S_\mathbf{c}(\btheta)|  \mathrm{d}\btheta.
\end{eqnarray}
Then bounding the contribution from any column which does not belong to $\mathfrak{D}_1, \ldots, \mathfrak{D}_{ T - 1  }$ trivially gives a bound for the integral over the minor arcs of
\[
X^{n- (T - 1) R - 1}
\sup_{\alpha \in [0,1] \setminus  \mathfrak{N}_{\mathfrak{L}'}}  \left|  \sum_{x \in \mathfrak{B} \cap [1,X]} e (\alpha x^d) \right|
\int_{[0,1]^R} \prod_{\ell =1}^{T - 1} \prod_{ \mathbf c \in \mathfrak D_\ell} \
\left \vert S_\mathbf{c}(\btheta) \right \vert \mathrm{d}\btheta.
\]
Applying H{\"o}lder's inequality this is bounded by
\[
X^{n- (T - 1) R - 1}
\sup_{\alpha \in [0,1] \setminus  \mathfrak{N}_{\mathfrak{L}'}  }  \left|  \sum_{x \in \mathfrak{B} \cap [1,X]} e (\alpha x^d) \right|
\prod_{\ell =1}^{T-1} \left(
\int_{[0,1]^R}  \prod_{ \mathbf c \in \mathfrak D_\ell} \
\left \vert S_\mathbf{c}(\btheta) \right \vert^{T-1} \mathrm{d}\btheta \right)^{1/(T - 1)} .
\]
Since the columns in $\mathfrak D_\ell$ are linearly independent, by  a linear change of variables  we obtain   
\[
\int_{[0,1]^R}  \prod_{ \mathbf c \in \mathfrak D_\ell} \
\left \vert S_\mathbf{c}(\btheta) \right \vert^{T-1} \mathrm{d}\btheta  \ll   \prod_{i=1}^R \int_0^1  \left|  \sum_{x \in \mathfrak{B} \cap [1,X]} e (\lambda_i x^d) \right|^{T-1} \mathrm{d} \lambda_i,
\]
for each $1 \leq \ell \leq T-1$.
We may now apply the bounds from Lemmas \ref{lem:weyl} and \ref{lem:mvt} or from Lemmas \ref{min+} and \ref{MVT1}, depending on $\mathfrak{B}$,                                                                                                                                                                                                                                                                                                                                                                                                                                                                                                                                                                                                                                                                                                                                                                                                                                                            to conclude the proof.
\end{proof}

\subsection{Major arc analysis}
We define
$$
\mathfrak{M}_\mathfrak{L}^+ = \bigcup_{1 \leq q \leq  \mathfrak{L} } \bigcup_{ \substack{ \a \in \ZZ^{R} \\ 0 \leq \a \leq q  \\  \gcd(q, \a) = 1  }  }
\{ \boldsymbol{\theta} \in [0,1]^{R}: \vert q \theta_i - a_i \vert < q \mathfrak{L} X^{ - d} \ (1 \leq i \leq R) \},
$$
which clearly satisfies $\mathfrak{M}_\mathfrak{L} \subseteq \mathfrak{M}_\mathfrak{L}^+$. For any $q \in \NN$, $a \in \ZZ$ and $\beta \in \RR$, we introduce the standard notation
\[
S(q,a)  = \sum_{1 \leq x \leq q} e_q(ax^d)
\quad \text{and} \quad
I(\beta) = \int_0^1 e(\beta \xi^d) \mathrm{d}\xi.
\]

\begin{lemma}
\label{lem0}
Suppose that $q \in \NN$, $a \in \ZZ$ and $\beta = \alpha - a/q$.
Then
$$
\sum_{1 \leq x \leq X } e(\alpha x^d) = X q^{-1} S(q,a) I(X^d \beta) + O \left(  \frac{q}{\gcd(q,a)}  (1 +  X^d |\beta|)  \right).
$$
\end{lemma}
\begin{proof}
The statement with the additional hypothesis $\gcd(q,a) = 1$ follows from \cite[Theorem 4.1]{Vaughan}.
Suppose $\gcd(q, a) = g$ and let $q_0 = q/g$ and $a_0 = a/g$.
Then
$$
q^{-1} S(q,a) =  q^{-1} \sum_{1 \leq x \leq q} e_q (a x^d) =  q^{-1} \sum_{1 \leq x \leq q} e_{q_0} (a_0 x^d) = q^{-1} g  \sum_{1 \leq x \leq q_0} e_{q_0} (a_0 x^d) =
q_0^{-1} S(q_0, a_0).
$$
Therefore, we see that we may remove the coprimality condition.
\end{proof}

For the smooth Weyl sums we have the following.
\begin{lemma}
\label{lem1}
Suppose that $1 \leq q \leq Z$, $a \in \ZZ$ and $\beta = \alpha - a/q$.
Then
$$
f(\alpha; X, Z) = q^{-1} S(q,a) w(\beta) + O \left(  \frac{q X}{\gcd(q,a) \log X}   (1 + X^d |\beta|) \right),
$$
where
$$
w(\beta) = \sum_{Z^d < m \leq X^d} \frac{1}{d} m^{ \frac{1}{d} - 1 } \varrho \left(  \frac{\log m}{ d \log Z}  \right)  e(\beta m)
$$
and $\varrho$ is the Dickman's function (for example, see \cite[pp.53]{Va89}).
\end{lemma}
\begin{proof}
The statement with the additional hypothesis $\gcd(q,a) = 1$ is precisely \cite[Lemma 5.4]{Va89}.
The coprimality condition may be removed in the same way as in the proof of Lemma \ref{lem0}.
\end{proof}

\begin{lemma}
Let $|\beta| < \mathfrak{L} X^{-d}$ and $w$ be as in Lemma \ref{lem1}.
Then
$$
w(\beta) = \varrho \left(  \frac{d \log X}{d \log Z}  \right) X  I( X^d  \beta )  + O  \left( \frac{X}{\log Z}  +  Z \right).
$$
\end{lemma}
\begin{proof}
Let us denote
$$
P(y) = \sum_{Z^d < m \leq y} \frac{1}{d} m^{\frac{1}{d} - 1} e(\beta m).
$$
Then, by summation by parts, it follows that
\begin{eqnarray}
\notag
w(\beta) &=& \sum_{Z^d < m \leq X^d} \frac{1}{d} m^{\frac{1}{d} - 1} e(\beta m) \varrho \left(  \frac{\log m}{d \log Z}  \right)
\\
\notag
&=&
P(X^d) \varrho \left(  \frac{d \log X}{d \log Z}  \right) + O \left(1 +  \int_{Z^d}^{X^d} |P(y)| \frac{ 1 }{y \log Z} \mathrm{d} y \right).
\end{eqnarray}
Since $|P(y)| \ll y^{\frac{1}{d}}$, we have
$$
\int_{Z^d}^{X^d} |P(y)| \frac{ 1 }{y \log Z} dy  \ll \frac{1}{\log Z} \int_{Z^d}^{X^d} y^{\frac{1}{d} - 1} \mathrm{d} y \ll \frac{X}{\log Z}.
$$
Therefore, we obtain
$$
w(\beta) = \varrho \left(  \frac{d \log X}{d \log Z}  \right) \sum_{1 \leq  m \leq X^d} \frac{1}{d} m^{\frac{1}{d} - 1} e(\beta m) + O  \left( \frac{X}{\log Z} + Z \right).
$$
By the mean value theorem, we obtain
\begin{eqnarray}
\frac{1}{d}
\sum_{1 \leq  m \leq X^d} m^{\frac{1}{d}-1 }  e ( \beta m  )
\notag
&=&
\frac{1}{d}
\int_{1}^{X^d} x^{\frac{1}{d}-1 }  e ( \beta  x ) \mathrm{d} x
+ O\left( 1 + \sum_{1 \leq  m \leq X^d} m^{\frac{1}{d}-1}(m^{-1} + |\beta|)   \right)
\\
\notag
&=&
\int_{0}^{X}   e ( \beta  t^d ) \mathrm{d} t
+ O( 1 )
\\
\notag
&=&
X \int_{0}^{1}  e ( X^d \beta  y^d ) \mathrm{d} y
+ O( 1 )
\\
\notag
&=&
X  I( X^d  \beta ) + O( 1 ).
\end{eqnarray}
\end{proof}

Let us now combine the above three lemmas in the following convenient manner.
\begin{lemma}
\label{lem33}
Let $\eta > 0$ be sufficiently small and
\begin{eqnarray}
\label{CB}
C_\mathfrak{B} = \begin{cases}
                   1 & \mbox{if } \mathfrak{B} = \NN, \\
                   \varrho(1/\eta) & \mbox{if } \mathfrak{B} = \mathcal{A}(X, X^{\eta}).
                 \end{cases}
\end{eqnarray}
Let $\delta, A > 0$ be sufficiently small.
Suppose that $0 \leq a \leq q \leq \mathfrak{L}$, $\beta = \alpha - a/q$ and $|\beta| < \mathfrak{L} X^{-d}$.
Then
$$\left \vert
\sum_{ x \in \mathfrak{B} \cap [1, X]}  e(\alpha x^d) - C_{\mathfrak{B}} X q^{-1} S(q,a) I(X^d \beta) \right \vert \ll
\begin{cases}
\mathfrak L^2 \phantom{pp}   &\mbox{if } \mathfrak B = \mathbb N,\\
\frac{X \mathfrak L^2}{\log X} &\mbox{if }\mathfrak B = \mathcal{A}(X, X^\eta).
\end{cases}
$$
\end{lemma}

We define the truncated singular series
$$
\mathfrak{S}(B)= \sum_{1 \leq q \leq B} q^{- n }  \sum_{ \substack{ 1 \leq  \a \leq  q \\ \gcd(q, \a) = 1  } } \prod_{\c \in \textnormal{Col}( M ) } S(q, \a\cdot\c) \cdot  e_q \left(  -   \sum_{i = 1}^{R} \mu_i a_i \right)
$$
for any $B \geq  1$, and the truncated  singular integral
$$
\mathfrak{I}(B) =  \int_{ |\bgamma| <  B }   \prod_{ \c \in \textnormal{Col}(M) } I(\bgamma \cdot  \c)  \cdot  e \left(   -  \frac{1}{X^d} \sum_{i = 1}^{R} \mu_i \gamma_i \right)  \mathrm{d} \bgamma
$$
for any $B > 0$.

\begin{proposition}\label{lem:major}
Let $\eta > 0$ be sufficiently small and $C_\mathfrak{B}$ as in (\ref{CB}).
Then
\[
\int_{\mathfrak{M}_\mathfrak{L} ^+} \prod_{\c \in \textnormal{Col}(M)} S_\c(\btheta) \cdot  e\left(-\sum_{i = 1}^{R}\mu_i \theta_i  \right)\mathrm{d}\btheta
=
C_\mathfrak{B}^n  X^{n - d R} \mathfrak{S}( \mathfrak L) \mathfrak{I}(\mathfrak L) + O( X^{n  - d R} \mathfrak{L}^{- 1}).
\]
\end{proposition}

\begin{proof}
First, if $\btheta \in \mathfrak{M}_\mathfrak{L}^+$ then there exist $0 \leq \a \leq q$ such that $\gcd(q, \a) = 1$ and
$$
\left| \c \cdot \btheta - \frac{ \c \cdot \a }{q} \right|  < C \mathfrak{L} X^{-d},
$$
where $C > 0$ is a constant depending only on $\c$; therefore, $\c \cdot \btheta$, reduced modulo $1$, satisfies the hypotheses of  Lemma \ref{lem33} with
$\c \cdot \bgamma$ and $C \mathfrak{L}$ in place of $\beta$ and $\mathfrak{L}$, respectively.
Thus we may apply Lemma  \ref{lem33} to $S_\c(\btheta)$ for any $\c \in \textnormal{Col}( M )$ and $\btheta \in \mathfrak{M}_\mathfrak{L}^+$.
The measure of $\mathfrak M_{\mathfrak L}^+$ is at most $\mathfrak L^{2R+1} X^{-dR}$ and thus integrating the error term
coming from applying Lemma \ref{lem33} to
$\prod_{\c \in \textnormal{Col}( M )} S_\c(\btheta)$
gives a total error of size
\[
\begin{cases}
O(X^{n-dR-1} \mathfrak{L}^{2R+3})  &\mbox{if } \mathfrak B = \mathbb N,\\
O(X^{n-dR} \frac{\mathfrak{L}^{2R+3}}{\log X}) &\mbox{if } \mathfrak B = \mathcal A(X, X^\eta).
\end{cases}
\]
The former case clearly provides a suitable error term for sufficiently small $\delta >0$ and in the latter case the error suffices on choosing $A < \frac{1}{2R+4}$.
As a result we have
\begin{eqnarray}
\label{major}
\notag
&&
\int_{  \mathfrak{M}^+_\mathfrak{L}  }  \prod_{\c \in \textnormal{Col}( M )} S_\c(\btheta) \cdot  e \left(-\sum_{i=1}^{R}  \mu_i  \theta_i    \right) \mathrm{d}\btheta
\\
\notag
&=& C_{\mathfrak{B}}^n X^n \sum_{1 \leq q \leq \mathfrak{L}}  q^{- n }  \sum_{ \substack{ 1 \leq  \a \leq  q \\ \gcd(q, \a) = 1  } }
 \prod_{\c \in \textnormal{Col}( M )} S_\c (\a /q)  \cdot e_q \left(-\sum_{i=1}^{R}   \mu_i  a_i   \right) \times
\\
\notag
&&\int_{\vert \bgamma \vert <  \mathfrak{L} X^{-d}}
\prod_{\c \in \textnormal{Col}( M )} I ( X^d \c\cdot\bgamma )  \cdot   e \left(-\sum_{i=1}^{R}  \mu_i  \gamma_i     \right) \mathrm{d}\bgamma
+ O \left(  X^{n  - d R} \mathfrak{L}^{- 1}\right)\\
\notag
&=&
C_{\mathfrak{B}}^n  X^n \mathfrak{S} (\mathfrak{L}) \int_{\vert \bgamma \vert <  \mathfrak{L} X^{-d}}
\prod_{\c \in \textnormal{Col}( M )} I (X^d \c\cdot \bgamma )  \cdot   e \left(-\sum_{i=1}^{R}  \mu_i  \gamma_i    \right) \mathrm{d}\bgamma
+ O ( X^{ n - d R} \mathfrak{L}^{ -1 }  )
\\
\notag
&=&
C_{\mathfrak{B}}^n  X^{n- d R }  \mathfrak{S} (\mathfrak{L}) \int_{\vert \bgamma \vert <  \mathfrak{L} }
\prod_{\c \in \textnormal{Col}( M )} I( \c\cdot \bgamma )  \cdot   e \left( - \frac{ 1 }{X^d}   \sum_{i=1}^{R}  \mu_i \gamma_i    \right) \mathrm{d}\bgamma
+ O ( X^{ n - d R} \mathfrak{L}^{ -1 } )
\\
\notag
&=&
C_{\mathfrak{B}}^n  X^{n - d R} \mathfrak{S} (\mathfrak{L}) \mathfrak{I} (\mathfrak{L}) + O ( X^{ n - d R} \mathfrak{L}^{ -1 }),
\end{eqnarray}
which completes the claim.
\end{proof}


\section{Singular series and singular integral}
Let us denote
$$
A (q) = q^{- n }  \sum_{ \substack{ 1 \leq  \a \leq  q \\ \gcd(q, \a) = 1  } } \prod_{\c \in \textnormal{Col}( M ) } S(q, \a\cdot\c) \cdot  e_q \left(  -   \sum_{i = 1}^{R} \mu_i a_i \right).
$$
We define the singular series as
\begin{eqnarray}
\label{singser}
\mathfrak{S}  = \sum_{q =1}^\infty A (q) = \lim_{B \rightarrow \infty} \mathfrak S(B).
\end{eqnarray}

In the following lemma, we bound the quantity $A(q)$ in order to show that the singular series does indeed converge absolutely.
\begin{lemma}
\label{lemA}
Let $q \in \mathbb{N}$.
Then
$$
A (q)  \ll  q^{- R( \frac{ T }{d} - 1  )}.
$$
\end{lemma}
\begin{proof}
By \cite[Lemma 6.4]{Davenport}, we have
$$
|S(q, \a.\c)| =  |S(q /\gcd(q, \a\cdot\c) ,  \a/ \gcd(q, \a\cdot\c) ) | \ll  \left( \frac{q}{\gcd(q, \a\cdot\c)}  \right)^{1 - \frac{1}{d} }.
$$
We know that there exist pairwise disjoint sets $\mathfrak D_1, \ldots \mathfrak D_T$ of $R$ linearly independent columns of $M$.
Applying H{\"o}lder's inequality, it follows that
\begin{eqnarray}
\label{Hod}
 |A(q)|  &\leq&  q^{  - T R } \sum_{ \substack{ 1 \leq  \a \leq  q \\ \gcd(q, \a) = 1  } }  \prod_{\ell = 1}^{T}  \prod_{\c \in \mathfrak{D}_\ell } |S(q, \a\cdot\c)|
\\
&\ll&
\notag
q^{  - T R }   \prod_{\ell = 1}^{ T }   \left( \sum_{ \substack{ 1 \leq  \a \leq  q \\ \gcd(q, \a) = 1  } }  \prod_{\c \in \mathfrak{D}_\ell }   |S(q, \a\cdot\c)|^T    \right)^{\frac{ 1 }{ T }}
\\
&\ll&
\notag
q^{ - \frac{T R }{d} }  \prod_{\ell = 1}^{ T }   \left( \sum_{ \substack{ 1 \leq  \a \leq  q \\ \gcd(q, \a) = 1  } }  \prod_{\c \in \mathfrak{D}_\ell }  \gcd(q, \a\cdot\c)^{- \frac{  (d-1)  T  }{d} }    \right)^{\frac{ 1 }{ T }}.
\end{eqnarray}
Let $1 \leq \ell \leq T$ and denote $\b =  M( \mathfrak{D}_\ell )^t \a$.
Then it is clear that $|\b| \ll q$. Since $M( \mathfrak{D}_\ell )^t$ is invertible, we have
\begin{eqnarray}
\sum_{ \substack{ 1 \leq  \a \leq  q \\ \gcd(q, \a) = 1  } }  \prod_{\c \in \mathfrak{D}_\ell }  \gcd(q, \a\cdot\c)^{- \frac{ (d-1) T }{d} }
&\ll&
\sum_{1 \leq \b \leq q}   \prod_{i = 1}^{R} \gcd(q, b_i)^{- \frac{ (d-1) T }{d}}
\notag
\\
\notag
&\ll&
\sum_{ \substack{  g_i | q  \\ 1 \leq i \leq R } }   (g_1 \cdots g_R)^{- \frac{(d-1) T}{d}}  \frac{ q^R }{g_1 \cdots g_{R}}
\\
\notag
&\ll& q^{R}.
\end{eqnarray}
The result follows on substituting this estimate into (\ref{Hod}).
\end{proof}

Using this lemma we may extend the truncated singular series.
\begin{lemma}
\label{SSSS}
Suppose
$
T > \frac{d (R + 1)}{R}.
$
Then
$$
\mathfrak{S}  =  \mathfrak{S} (B) +  O( B^{1 - R (\frac{T}{d} - 1)})
$$
for any $B \geq 1$. In fact,
$$
\mathfrak{S} = \prod_{p \textnormal{ prime} } \chi(p),
$$
where
$$
\chi(p) = 1 + \sum_{k = 1}^{\infty} A (p^k).
$$
\end{lemma}

\begin{proof}
The statement is obtained by Lemma \ref{lemA} as follows
$$
| \mathfrak{S}   -  \mathfrak{S} (B)| \leq \sum_{q >  B} |A (q)| \ll \sum_{q >  B} q^{- R (\frac{T}{d} - 1)} \ll B^{1 - R (\frac{T}{d} - 1)}.
$$
Since $A (q_1 q_2) = A (q_1) A (q_2)$ for any coprime positive integers $q_1$ and $q_2$, we also have
$$
\mathfrak{S} = \prod_{p \textnormal{ prime} } \chi(p)
$$
as desired.
\end{proof}

Similarly, we define the singular integral as
\begin{eqnarray}
\label{singint}
\mathfrak{I} =  \int_{\RR^{R}}   \prod_{ \c \in \textnormal{Col}(M) } I(\bgamma \cdot \c)  \cdot  e \left(   -  \frac{1}{X^d} \sum_{i = 1}^{R} \mu_i \gamma_i \right)  \mathrm{d} \bgamma = \lim_{B \rightarrow \infty} \mathfrak I(B).
\end{eqnarray}
We may also extend the truncated singular integral.
\begin{lemma}\label{lem:singint}
Suppose $T > d$.
Then
\[
\mathfrak{I} (B) =  \mathfrak{I}  +    O(B^{1 - \frac{ T }{ d  } } )
\]
for any $B > 1$.
\end{lemma}

\begin{proof}
We begin with the bound
\begin{eqnarray}
\label{bdd}
I(\bgamma \cdot \c) = \int_0^1 e(\bgamma \cdot \c \xi^d )\mathrm{d} \xi
\ll
\min \{ 1, \vert \bgamma \cdot \c \vert^{-1/d} \},
\end{eqnarray}
which for instance can be found in \cite[p. 21]{Davenport} or \cite[Lemma 2.8]{Vaughan}.
We know that there exist pairwise disjoint sets $\mathfrak D_1, \ldots \mathfrak D_T$ of $R$ linearly independent columns of $M$.
It then follows by  H{\"o}lder's inequality that
\begin{eqnarray}
\label{III}
|\mathfrak{I}  - \mathfrak{I}  (B) | &\leq &
\int_{\vert \bgamma \vert \geq B} \prod_{\c \in \textnormal{Col}(M)} \min \{ 1,  \vert \bgamma \cdot \c \vert^{-1/d}\} \mathrm{d}\bgamma
\\
\notag
&\leq&
\int_{\vert \bgamma \vert \geq B}  \prod_{\ell = 1}^{ T }  \prod_{\c \in \mathfrak{D}_\ell}  \min \{ 1,  \vert \bgamma \cdot  \c \vert^{-1/d}\} \mathrm{d}\bgamma
\\
\notag
&\leq&
\prod_{\ell = 1}^{ T } \left(  \int_{\vert \bgamma \vert \geq B} \prod_{\c \in \mathfrak{D}_\ell}  \min \{ 1,  \vert \bgamma \cdot  \c \vert^{-1/d}\}^{ T } \mathrm{d}\bgamma \right)^{ \frac{1}{T} }.
\end{eqnarray}
By the change of variable $\widetilde{\bgamma} = M(\mathfrak{D}_\ell)^t \bgamma $, we obtain
\begin{eqnarray}
&&\int_{\vert \bgamma \vert \geq B} \prod_{\c \in \mathfrak{D}_\ell}  \min \{ 1,  \vert \bgamma \cdot  \c \vert^{-1/d}\}^{ T } \mathrm{d}\bgamma
\notag
\\
&\ll&
\notag
\int_{\vert \widetilde{\bgamma} \vert \gg B} \min \{ 1,  \vert \widetilde{\gamma}_1 \vert^{-1/d}\}^{ T } \cdots  \min \{ 1,  \vert \widetilde{\gamma}_R \vert^{-1/d}\}^{ T }   \mathrm{d} \widetilde{ \bgamma }
\\
\notag
&\ll&
\int_{ \substack{ \widetilde{\gamma}_R > \cdots >  \widetilde{\gamma}_1 \geq 0 \\ \widetilde{\gamma}_R \gg B } }
\min \{ 1,  \vert \widetilde{\gamma}_1 \vert^{-1/d}\}^{ T } \cdots  \min \{ 1,  \vert \widetilde{\gamma}_R \vert^{-1/d}\}^{ T }   \mathrm{d} \widetilde{ \bgamma }
\\
\notag
&\ll& B^{1 - \frac{ T }{d}}
\end{eqnarray}
for each $1 \leq \ell \leq T$. On substituting this estimate into (\ref{III}), it follows that
$$
|\mathfrak{I}  - \mathfrak{I} (B) | \ll B^{1 - \frac{ T }{d}}.
$$

\end{proof}

We may now conclude the proof of our main results.
\begin{proof}[Proof of Theorems \ref{main2} and \ref{main2+}]
Recall our starting point for the circle method (\ref{circle...}) and that $\mathfrak{M}_\mathfrak{L} \subseteq \mathfrak{M}_\mathfrak{L}^+$.
On combining Propositions \ref{prop:minor} and \ref{lem:major}, we have
\[
N(\mathfrak{B};X)
=
C_\mathfrak{B}^n  X^{n - d R} \mathfrak{S}( \mathfrak L) \mathfrak{I}(\mathfrak L) + O( X^{n  - d R} \mathfrak{L}^{- \gamma}),
\] for some $\gamma > 0$.
Lastly, we obtain from Lemmas \ref{SSSS} and \ref{singint} that
$$
\mathfrak{S} (\mathfrak{L}) \mathfrak{I} (\mathfrak{L})  =  \mathfrak{S}  \mathfrak{I}  +  O \left(  \mathfrak{L}^{ 1 - R (\frac{T}{d} - 1) } + \mathfrak{L}^{1 - \frac{T}{d} } \right).
$$
These two equations together give the desired asymptotic formula.
\end{proof}

\bibliographystyle{rome}
\bibliography{diagII}

\end{document}